\documentclass[12pt]{article}
\usepackage[utf8]{inputenc}
\usepackage[T1]{fontenc}
\usepackage{lmodern}
\usepackage{amsmath,amssymb,amsthm}
\usepackage{geometry}
\usepackage{indentfirst}
\usepackage{booktabs}
\usepackage{xcolor}
\usepackage{hyperref}
\newtheorem{theorem}{Theorem}
\newtheorem{lemma}{Lemma}
\newtheorem{proposition}{Proposition}
\newtheorem{corollary}{Corollary}
\theoremstyle{remark}
\newtheorem{remark}{Remark}
\theoremstyle{definition}
\newtheorem{definition}{Definition}

\title{On the Rank Condition for Perfect Cuboids}
\author{Valery Asiryan\\[2mm]
\small\texttt{asiryanvalery@gmail.com}}
\date{\small October 7, 2026}

\begin{document}
\maketitle

\begin{abstract}
Paulsen and West associate perfect cuboids with triples of rational points on a congruent-number elliptic curve and assert that the corresponding curve has rank at least two. The proof of their Theorem~4.2 contains a gap: translation of a generator by a point of order two interchanges only the odd-indexed terms of two coordinate sequences, rather than the entire sequences. We exhibit two explicit counterexamples to that intermediate assertion and give a complete proof of the rank obstruction. The division-polynomial argument reduces a hypothetical rank-one configuration to one remaining parity pattern. Ingram's primitive-divisor theorem then forces an additive relation among its indices, which is incompatible with a positivity lemma on the elliptic curve. More precisely, the three points of every positive cuboid configuration generate a subgroup of rank at least two, even when the ambient curve has larger rank. The proof is unconditional.


\end{abstract}


\section{Introduction}\label{sec:introduction}

A perfect cuboid is a rectangular box whose three edges, three face diagonals, and space diagonal all have positive integer lengths. Its existence is equivalent to the existence of such a box with positive rational lengths, since denominators can be cleared simultaneously.

Paulsen and West~\cite{PW} reformulate this problem in terms of rational points on congruent-number elliptic curves. Their Proposition~3.1 describes the relevant triples, and their Theorem~4.2 asserts that the associated curve must have Mordell--Weil rank at least two. The proof of the latter statement uses two sequences of positive abscissas and a change of generator intended to interchange them. That interchange does not hold at even indices. The subsequent argument therefore leaves a parity case untreated.

\paragraph{Goal and method.}
The purpose of this article is to supply the missing argument and to state its precise consequence for the original cuboid problem. The valid division-polynomial part of the Paulsen--West proof is retained and its scope is made explicit. The remaining case is resolved by combining a primitive denominator prime with a geometric restriction on sums of three positive points. This also gives a statement about the subgroup generated by the three points, without assuming that it is the whole Mordell--Weil group.

Let $N$ be a positive squarefree integer. We write
\begin{equation}\label{eq:models}
 D_N:\ N\eta^2=t-t^3,
 \qquad E_N:\ Y^2=X^3-N^2X.
\end{equation}
The identity on both curves is the point $O$ at infinity, and the rational isomorphism is
\begin{equation}\label{eq:isomorphism}
 \iota:D_N\longrightarrow E_N,
 \qquad (t,\eta)\longmapsto(-Nt,N^2\eta).
\end{equation}
On $D_N$ put
\begin{equation}\label{eq:torsion}
 S=(0,0),\qquad T=(-1,0),\qquad U=(1,0)=S+T.
\end{equation}
The standard torsion calculation for congruent-number curves gives
\[
 D_N(\mathbb{Q})_{\rm tors}=\{O,S,T,U\}\simeq(\mathbb{Z}/2\mathbb{Z})^2;
\]
see, for example, the treatment of congruent numbers in Koblitz~\cite{Koblitz}.

\begin{definition}\label{def:configuration}
A \emph{positive cuboid configuration} on $D_N$ is a triple of rational points $Q_i=(t_i,\eta_i)$ such that
\begin{equation}\label{eq:compatibility}
 0<t_i<1\quad(i=1,2,3),
 \qquad t_1t_2+t_1t_3+t_2t_3=1.
\end{equation}
No condition on the signs of the ordinates is imposed.
\end{definition}

\begin{theorem}\label{thm:main}
For every positive cuboid configuration on $D_N$, the subgroup
\[
 H=\langle Q_1,Q_2,Q_3\rangle\subseteq D_N(\mathbb{Q})
\]
has rank at least two. In particular, $\operatorname{rank} E_N(\mathbb{Q})\ge2$.
\end{theorem}

Here and throughout, the rank of a finitely generated abelian group $H$ means $\dim_{\mathbb{Q}}(H\otimes_{\mathbb{Z}}\mathbb{Q})$. Theorem~\ref{thm:main} is proved in Sections~\ref{sec:positivity}--\ref{sec:completion}. The ambient-rank assertion is the conclusion stated by Paulsen and West; the additional precision is that even a subgroup of rank one inside a curve of higher rank cannot support such a triple.


\section{The consequence for a perfect cuboid}\label{sec:cuboid}

We record the map from a cuboid explicitly, so that the invariant to which the rank condition applies is fixed without a normalization ambiguity. For $q\in\mathbb{Q}_{>0}$, let $\operatorname{sf}(q)$ denote the unique positive squarefree integer representing $q$ in $\mathbb{Q}^{\times}/\mathbb{Q}^{\times2}$.

\begin{proposition}\label{prop:cuboid-map}
Suppose $a,b,c,d,e,f,g\in\mathbb{Q}_{>0}$ satisfy
\begin{equation}\label{eq:cuboid}
 \begin{gathered}
 d^2=a^2+b^2,\qquad e^2=a^2+c^2,\qquad f^2=b^2+c^2,\\
 g^2=a^2+b^2+c^2.
 \end{gathered}
\end{equation}
Set $N=\operatorname{sf}(abcg/2)$ and choose $s\in\mathbb{Q}_{>0}$ with $abcg/2=Ns^2$. Then the points
\begin{equation}\label{eq:cuboid-points}
 \begin{aligned}
 Q_1&=\left(\frac{ag}{de+bc},\frac{2s}{de+bc}\right),\\
 Q_2&=\left(\frac{bg}{df+ac},\frac{2s}{df+ac}\right),\\
 Q_3&=\left(\frac{cg}{ef+ab},\frac{2s}{ef+ab}\right)
 \end{aligned}
\end{equation}
form a positive cuboid configuration on $D_N$.
\end{proposition}

\begin{proof}
Put $h=de+bc$. The identity
\[
 d^2e^2-b^2c^2=a^2g^2
\]
gives $h^2-a^2g^2=2bch>0$. Hence $0<ag/h<1$, and
\[
 \frac{ag}{h}-\left(\frac{ag}{h}\right)^3
 =\frac{2abcg}{h^2}
 =N\left(\frac{2s}{h}\right)^2.
\]
The other two points are treated cyclically.

To check the last equation in \eqref{eq:compatibility}, consider the triangle with side lengths $af$, $be$, and $cd$. These are the face diagonals of the box with edges $ab$, $ac$, and $bc$, and they satisfy the strict triangle inequalities. For example,
\[
 (be)^2+(cd)^2-(af)^2=2b^2c^2>0,
\]
and cyclically, so this triangle is acute. If $\alpha_1$ is the angle opposite $af$, the cosine rule gives
\[
 \cos\alpha_1=\frac{bc}{de},
 \qquad \sin\alpha_1=\frac{ag}{de}.
\]
Thus $t_1=\tan(\alpha_1/2)$, and cyclically $t_i=\tan(\alpha_i/2)$. Since $\alpha_1+\alpha_2+\alpha_3=\pi$, the cosine addition formula for the three half-angles yields
\[
 0=\prod_{i=1}^3\cos(\alpha_i/2)
       \bigl(1-t_1t_2-t_1t_3-t_2t_3\bigr).
\]
All three cosine factors are positive. This proves the required identity. The same triangle has area $abcg/2$, which also explains the normalization of $N$.
\end{proof}

\begin{corollary}\label{cor:cuboid-rank}
Every positive rational perfect cuboid determines a squarefree integer
\[
 N=\operatorname{sf}(abcg/2)
\]
for which $\operatorname{rank} E_N(\mathbb{Q})\ge2$. More precisely, the three points in \eqref{eq:cuboid-points} generate a subgroup of rank at least two.
\end{corollary}

This follows from Proposition~\ref{prop:cuboid-map} and Theorem~\ref{thm:main}. The converse correspondence between configurations and cuboids is given in Paulsen--West~\cite[Proposition~3.1]{PW}; the rank obstruction proved here needs only the direction established above.

\begin{remark}\label{rem:rank-zero}
Rank zero is excluded before any of the rank-one argument is used. Indeed, every point in a positive configuration has $0<t<1$ and is therefore non-torsion by \eqref{eq:torsion}. A rank-zero congruent-number curve has no such point.
\end{remark}


\section{The gap in the sequence interchange}\label{sec:gap}

For a non-torsion point $G\in D_N(\mathbb{Q})$ with $0<t(G)<1$, define point sequences
\begin{equation}\label{eq:sequences}
 \mathcal A_k(G)=[k]G+\epsilon(k)S,
 \qquad \mathcal B_k(G)=\mathcal A_k(G)+T,
 \qquad
 \epsilon(k)=\begin{cases}0,&k\text{ odd},\\1,&k\text{ even}.
 \end{cases}
\end{equation}
Write $A_k(G)=t(\mathcal A_k(G))$ and $B_k(G)=t(\mathcal B_k(G))$ for their abscissas. These are the two scalar sequences denoted by $A$ and $B$ in Paulsen--West.

The torsion translations are
\begin{equation}\label{eq:translations}
 \begin{aligned}
 (t,\eta)+S&=\left(-\frac1t,\frac{\eta}{t^2}\right),\\
 (t,\eta)+T&=\left(\frac{1-t}{1+t},-\frac{2\eta}{(1+t)^2}\right).
 \end{aligned}
\end{equation}
In particular, translation by $T$ preserves the interval $(0,1)$ for abscissas.

In the proof of Theorem~4.2, Paulsen and West~\cite{PW} replace $G$ by $G+T$ and assert that ``the sequences $A$ and $B$ would trade places.'' This is used to assume that an odd number of the three entries belong to $B$. The exact transformation is instead the following.

\begin{proposition}\label{prop:actual-swap}
Under the replacement $G\mapsto G+T$, the entries satisfy
\begin{equation}\label{eq:actual-swap}
 \begin{aligned}
 (A_k(G+T),B_k(G+T))
 &=\begin{cases}
 (B_k(G),A_k(G)),&k\text{ odd},\\
 (A_k(G),B_k(G)),&k\text{ even}.
 \end{cases}
 \end{aligned}
\end{equation}
\end{proposition}
\begin{proof}
Since $[2]T=O$,
\[
 [k](G+T)=[k]G+[k]T
 =\begin{cases}[k]G+T,&k\text{ odd},\\ {}[k]G,&k\text{ even}.
 \end{cases}
\]
Substitution in \eqref{eq:sequences} proves the claim, first for the points and then for their abscissas.
\end{proof}

In fact, $A_k(G)$ and $B_k(G)$ are always distinct: equality would imply $A_k(G)^2+2A_k(G)-1=0$, which has no rational root. Thus the failure at even indices occurs for every rational non-torsion base point, including any choice of generator in the rank-one setting.

\subsection{Two explicit counterexamples}

On $D_6$, take
\[
 G=\left(\frac12,\frac14\right),
 \qquad G+T=\left(\frac13,-\frac29\right).
\]
Both points satisfy $6\eta^2=t-t^3$. Doubling gives
\[
 [2]G=\left(-\frac{25}{24},-\frac{35}{288}\right)
       =[2](G+T).
\]
Using \eqref{eq:translations}, for \emph{both} choices of base point one obtains
\begin{equation}\label{eq:example-six}
 A_2=\frac{24}{25},\qquad B_2=\frac1{49}.
\end{equation}
The two values are distinct and remain in their respective sequences.

A second example is provided by $D_5$:
\[
 G=\left(\frac45,\frac6{25}\right),
 \qquad G+T=\left(\frac19,-\frac4{27}\right).
\]
Here
\[
 t([2]G)=t([2](G+T))=-\frac{1681}{720},
\]
and the corresponding unchanged second entries are
\begin{equation}\label{eq:example-five}
 A_2=\frac{720}{1681},\qquad B_2=\frac{961}{2401}.
\end{equation}
These computations use exact rational arithmetic; they can also be checked directly from
\begin{equation}\label{eq:duplication-t}
 t([2]G)=-\frac{(1+t(G)^2)^2}{4N\eta(G)^2}.
\end{equation}

\subsection{What the counterexamples do and do not establish}

The examples contradict the asserted interchange of the entire sequences. They are not counterexamples to the rank statement and are not claimed to be cuboid configurations. Their role is to identify the precise invalid reduction in the proof.

If a triple has $b$ entries from $B$ and $o$ odd indices, then changing $G$ to $G+T$ changes the parity of $b$ by $o$ modulo two. Consequently, when $o$ is even, a triple with an even number of $B$ entries remains of that kind. In particular, a pattern consisting of an odd-indexed $A$ entry, an odd-indexed $B$ entry, and an even-indexed $B$ entry is not reduced to the odd-$B$ case by this operation. Section~\ref{sec:reduction} proves that this is precisely the pattern still requiring an argument after the valid polynomial exclusions have been applied.


\section{A geometric restriction on positive triples}\label{sec:positivity}

The following lemma is the geometric ingredient needed to exclude that last pattern. We use the standard identification of an elliptic curve with its degree-zero divisor class group; see Silverman~\cite[Chapter~III]{Silverman}.

\begin{lemma}\label{lem:positivity}
Let $Q_1,Q_2,Q_3$ be a positive cuboid configuration on $D_N$. For every choice of signs $\sigma_i\in\{1,-1\}$,
\begin{equation}\label{eq:forbidden-sum}
 \sigma_1Q_1+\sigma_2Q_2+\sigma_3Q_3\ne S.
\end{equation}
\end{lemma}

\begin{proof}
Negating a point changes only its ordinate. Thus it suffices to rule out $Q_1+Q_2+Q_3=S$ for arbitrary choices of the ordinates.

Suppose that this equality holds. Since $S=-S$, the divisor
\[
 (S)+(Q_1)+(Q_2)+(Q_3)-4(O)
\]
is principal. There is therefore a rational function with this divisor. The pole orders of $t$ and $\eta$ at $O$ are two and three, respectively, and the Riemann--Roch space is
\[
 L(4O)=\langle 1,t,\eta,t^2\rangle_{\mathbb{Q}}.
\]
The function vanishes at $S$, so its constant term is zero. Its $\eta$ coefficient cannot vanish: a polynomial in $t$ which vanishes at $S$ has even vanishing order there, whereas the prescribed multiplicity of $S$ is one. Its $t^2$ coefficient cannot vanish either, since its pole has exact order four. After scaling, the function is consequently
\begin{equation}\label{eq:parabola}
 \eta-\alpha t^2-\beta t,
 \qquad \alpha,\beta\in\mathbb{Q},\quad\alpha\ne0.
\end{equation}

Substitution of $\eta=\alpha t^2+\beta t$ in the curve equation, followed by removal of the root $t=0$ corresponding to $S$, gives
\begin{equation}\label{eq:cubic-positivity}
 N\alpha^2t^3+(2N\alpha\beta+1)t^2+N\beta^2t-1=0.
\end{equation}
Its roots, counted with multiplicity, are $t_1,t_2,t_3$. Write $s_1,s_2,s_3$ for their elementary symmetric functions. Vieta's formulas give
\begin{equation}\label{eq:vieta}
 s_1=-\frac{2N\alpha\beta+1}{N\alpha^2},
 \qquad s_2=\frac{\beta^2}{\alpha^2},
 \qquad s_3=\frac1{N\alpha^2}.
\end{equation}
By compatibility, $s_2=1$, so $\beta=\alpha$ or $\beta=-\alpha$.

If $\beta=\alpha$, then $s_1=-2-1/(N\alpha^2)<0$, contrary to $t_i>0$. If $\beta=-\alpha$, then $s_1=2-1/(N\alpha^2)$, and
\[
 \prod_{i=1}^3(1-t_i)=1-s_1+s_2-s_3=0.
\]
This contradicts $t_i<1$. The divisor argument includes repeated points through their multiplicities, so no distinctness assumption on the original triple is needed.
\end{proof}


\section{Reduction of a rank-one configuration}\label{sec:reduction}

Assume for a contradiction that the subgroup $H$ in Theorem~\ref{thm:main} has rank at most one. Every $Q_i$ is non-torsion, so $H$ has rank exactly one. In the free abelian group
\[
 D_N(\mathbb{Q})/D_N(\mathbb{Q})_{\rm tors},
\]
the saturation of the image of $H$ is an infinite cyclic group. Choose a lift $G\in D_N(\mathbb{Q})$ of a generator of that saturation. There are nonzero integers $k_i$ and rational two-torsion points $R_i$ such that $Q_i=[k_i]G+R_i$. Negating individual $Q_i$, if necessary, makes all $k_i$ positive without changing \eqref{eq:compatibility}.

The real locus of $D_N$ has two components: its identity component has $t\le-1$ together with $O$, and the other has $0\le t\le1$. Translation by $S$ interchanges them. We may therefore translate $G$ by rational two-torsion to arrange
\begin{equation}\label{eq:base}
 z=t(G)\in\mathbb{Q},\qquad 0<z<1.
\end{equation}
This changes only the torsion terms $R_i$. An odd multiple of $G$ has $0<t<1$, and an even multiple has $t<-1$; none is torsion. Hence, for each positive index $k$, exactly the two translates $\mathcal A_k(G)$ and $\mathcal B_k(G)$ in \eqref{eq:sequences} have abscissa in $(0,1)$. The three points under consideration can thus be taken from these point sequences.

The standard division polynomials give
\begin{equation}\label{eq:multiplication-polynomials}
 t([k]G)=\frac{P_k(z)}{Q_k(z)},\qquad P_k,Q_k\in\mathbb{Z}[z],
\end{equation}
with
\begin{equation}\label{eq:degrees}
 \deg P_k=k^2,\quad\operatorname{lc}(P_k)=1,
 \qquad\deg Q_k=k^2-1,\quad\operatorname{lc}(Q_k)=k^2.
\end{equation}
For clarity, these polynomials are independent of $N$. Over a quadratic extension, $D_N$ becomes $v^2=t^3-t$ by changing the ordinate and keeping $t$ fixed. Multiplication commutes with this isomorphism. For $v^2=t^3-t$, the usual formula is
\[
 t([k]P)=\frac{t\psi_k^2-\psi_{k-1}\psi_{k+1}}{\psi_k^2}.
\]
After eliminating $v^2=t^3-t$, both numerator and denominator belong to $\mathbb{Z}[t]$; the division-polynomial recursion gives the leading coefficients in \eqref{eq:degrees}. This is the standard multiplication-polynomial calculation underlying Paulsen--West Lemma~4.1; see also Silverman~\cite[Chapter~III]{Silverman}. For instance,
\[
 P_1=z,\quad Q_1=1,\qquad
 P_2=(z^2+1)^2,\quad Q_2=4z(z^2-1).
\]

Combining \eqref{eq:multiplication-polynomials} with \eqref{eq:translations} yields
\begin{equation}\label{eq:AB-table}
 \begin{array}{c|cc}
 &k\ \text{odd}&k\ \text{even}\\[2pt]\hline
 A_k& P_k/Q_k&-Q_k/P_k\\[3pt]
 B_k&(Q_k-P_k)/(Q_k+P_k)&(P_k+Q_k)/(P_k-Q_k).
 \end{array}
\end{equation}
None of the displayed denominators vanishes at $z=t(G)$, since such a zero would make a multiple of $G$ a torsion point or $O$.

\begin{lemma}\label{lem:oddB}
A compatible triple selected from these sequences cannot have an odd number of entries from $B$.
\end{lemma}
\begin{proof}
Write its abscissas as $n_i(z)/d_i(z)$, using the expressions in \eqref{eq:AB-table}. Clearing denominators in \eqref{eq:compatibility} gives $Z(z)=0$, where
\begin{equation}\label{eq:cleared}
 Z=n_1n_2d_3+n_1n_3d_2+n_2n_3d_1-d_1d_2d_3
 \in\mathbb{Z}[z].
\end{equation}

Suppose first that exactly one entry, say the third, is from $B$. For each of the two $A$ entries, one of the numerator and denominator has degree one greater than the other and leading coefficient $1$ or $-1$. Both $B$ factors have the same degree and leading coefficient $1$ or $-1$. In the four summands of \eqref{eq:cleared}, each of the four possible choices of numerator or denominator for the first two positions occurs exactly once. Thus exactly one summand uses both larger-degree $A$ factors. It has strictly largest degree and leading coefficient $1$ or $-1$. Consequently $Z$ has leading coefficient $1$ or $-1$.

If all three entries are from $B$, multiply each numerator-denominator pair by a sign so that its numerator is monic. Let $\delta_i\in\{1,-1\}$ be the denominator leading coefficients. All four summands in \eqref{eq:cleared} have the same degree, with combined leading coefficient
\[
 \delta_1+\delta_2+\delta_3-\delta_1\delta_2\delta_3\in\{2,-2\}.
\]
Also $n_i\equiv d_i\pmod2$ as polynomials. Reducing \eqref{eq:cleared} modulo two therefore gives $Z\equiv0$. It follows that $Z/2\in\mathbb{Z}[z]$ has leading coefficient $1$ or $-1$.

In both cases, the rational-root theorem forces the rational number $z$ to be an integer. This contradicts \eqref{eq:base}.
\end{proof}

This is the part of the polynomial argument of Paulsen--West that remains valid. What is needed next is a reduction which respects Proposition~\ref{prop:actual-swap}.

\begin{lemma}\label{lem:residual}
If a positive cuboid configuration generates a subgroup of rank one, then, after the allowed changes of signs, base point, and labels, it has the form
\begin{equation}\label{eq:residual}
 \begin{aligned}
 Q_1&=[a]G,\\
 Q_2&=[b]G+T,\\
 Q_3&=[c]G+S+T,
 \end{aligned}
 \qquad a,b\text{ positive odd},\quad c\text{ positive even},
\end{equation}
where $G$ is non-torsion and $0<t(G)<1$.
\end{lemma}

\begin{proof}
Let $b_0$ be the number of entries from $B$ and $o$ the number of odd indices. Lemma~\ref{lem:oddB} makes $b_0$ even. If $o$ were odd, changing $G$ to $G+T$ would make $b_0$ odd, again contradicting that lemma. Thus $o$ is also even.

Suppose $b_0=0$. If all three indices are even, the unique term of largest degree in \eqref{eq:cleared} is $-d_1d_2d_3$, with leading coefficient $-1$. If precisely two indices are odd, the unique term of largest degree is the one containing those two odd numerators and the even denominator; it has leading coefficient $1$. The rational-root theorem excludes both possibilities. Therefore $b_0=2$.

If all indices are even, replace the base point by
\[
 G'=[2]G+S.
\]
It is non-torsion and has abscissa in $(0,1)$. Its sequences satisfy the point identities
\begin{equation}\label{eq:halving-indices}
 \mathcal A_k(G')=\mathcal A_{2k}(G),
 \qquad \mathcal B_k(G')=\mathcal B_{2k}(G).
\end{equation}
Indeed, $[k]G'=[2k]G+[k]S$, and $k+\epsilon(k)$ is odd. Thus all three indices can be divided by two while preserving their sequence labels. Repeating this operation finitely many times yields at least one odd index. Reapplying the preceding exclusions then shows that there are exactly two odd indices and two $B$ entries.

If both odd indices were $B$ entries, the replacement $G\mapsto G+T$ would produce three $A$ entries, which have already been excluded. The odd entries must therefore be one $A$ and one $B$, while the even entry is $B$. This is \eqref{eq:residual}.
\end{proof}

No step requires the new base point to generate the whole Mordell--Weil group. In particular, the index-halving operation in \eqref{eq:halving-indices} is valid for the stronger subgroup formulation of Theorem~\ref{thm:main}.


\section{Primitive divisors and completion of the proof}\label{sec:completion}

\subsection{Removing the torsion translations}

Use \eqref{eq:isomorphism} and continue to denote the image of the base point by $G$. Set
\[
 X_k=X([k]G),\qquad X(G)=-Nt(G)<0.
\]
For the pattern \eqref{eq:residual}, put
\[
 u=-X_a/N,\qquad v=-X_b/N,\qquad w=-X_c/N.
\]
The three positive abscissas on $D_N$ are
\[
 t_1=u,\qquad t_2=\frac{1-v}{1+v},\qquad
 t_3=\frac{w+1}{w-1}.
\]
Direct expansion gives
\[
 t_1t_2+t_1t_3+t_2t_3-1
 =\frac{2(uv+uw-vw+1)}{(1+v)(w-1)}.
\]
The denominator is nonzero. Hence compatibility is equivalent to
\begin{equation}\label{eq:raw}
X_aX_b+X_aX_c-X_bX_c+N^2=0.
\end{equation}

If $a=b$, equation~\eqref{eq:raw} would give $X_a^2+N^2=0$, which is impossible. Since $c$ has opposite parity to $a,b$, the three indices are pairwise distinct. Their maximum $m$ is therefore strict and is at least three.

If $m=a$, replace $G$ on $D_N$ by $G+T$ and interchange the two odd roles. The old odd $B$ entry becomes the new odd $A$ entry, the old odd $A$ entry becomes the new odd $B$ entry, and the even entry is unchanged. The pattern \eqref{eq:residual} is preserved, as is $0<t(G)<1$, and the maximum now has the $b$ label. We may consequently assume
\begin{equation}\label{eq:maximum}
 m=b\quad\text{or}\quad m=c.
\end{equation}

\subsection{The primitive-divisor input}

Write the multiples in lowest terms as
\begin{equation}\label{eq:denominators}
 X([k]G)=\frac{M_k}{D_k^2},\qquad
 M_k\in\mathbb{Z},\quad D_k\in\mathbb{Z}_{>0},\quad\gcd(M_k,D_k)=1.
\end{equation}
A prime divisor of $D_m$ is \emph{primitive} if it divides none of $D_1,\ldots,D_{m-1}$. We use the following specialization of Ingram's result.

\begin{theorem}[Ingram, {\cite[Lemma~7]{Ingram}}]\label{thm:Ingram}
Let $N>0$ be squarefree and let $G\in E_N(\mathbb{Q})$ be a non-torsion point with $X(G)<0$. For every $m\ge3$, the denominator $D_m$ has a primitive prime divisor $p$ at which $E_N$ has good reduction. In particular, $p\nmid2N$.
\end{theorem}

This is the negative-abscissa conclusion of Lemma~7 in~\cite[pp.~617--618]{Ingram}; it applies to every positive squarefree $N$.

At a good prime, $p\mid D_k$ is equivalent to $[k]\overline G=O$ in $E_N(\mathbb{F}_p)$. Thus primitivity at $m$ means that the reduced point $\overline G$ has exact order $m$.

\subsection{The forced relation among the indices}

Apply Theorem~\ref{thm:Ingram} to the maximum $m$ in \eqref{eq:maximum}. Then
\[
 v_p(X_m)<0,
 \qquad v_p(X_k)\ge0\quad(1\le k<m),
\]
and $\overline G$ has order $m$.

If $m=b$, rewrite \eqref{eq:raw} as
\[
 X_b(X_a-X_c)=-(X_aX_c+N^2).
\]
The right side is $p$-integral, whereas $X_b$ has negative valuation. It follows that $X_a-X_c$ has positive valuation, and hence
\[
 X_a\equiv X_c\pmod p.
\]
The smaller-index points reduce to finite points. On a nonsingular Weierstrass curve over a field of odd characteristic, equal abscissas imply that the points are equal or opposite. Therefore
\[
 [a]\overline G=\pm[c]\overline G,
 \qquad b\mid a-c\quad\text{or}\quad b\mid a+c.
\]
Since $0<a,c<b$ and $a\ne c$, the first divisibility is impossible. The second, together with $0<a+c<2b$, gives
\begin{equation}\label{eq:b-sum}
 b=a+c.
\end{equation}

If $m=c$, use instead
\[
 X_c(X_a-X_b)=-(X_aX_b+N^2).
\]
Exactly the same argument gives $[a]\overline G=\pm[b]\overline G$. The inequalities $0<a,b<c$ and $a\ne b$ then force
\begin{equation}\label{eq:c-sum}
 c=a+b.
\end{equation}

\subsection{Contradiction with positivity}

Return to the points on $D_N$ in \eqref{eq:residual}. If \eqref{eq:b-sum} holds, then
\[
 Q_1-Q_2+Q_3=[a-b+c]G-T+(S+T)=S.
\]
If \eqref{eq:c-sum} holds, then
\[
 Q_1+Q_2-Q_3=[a+b-c]G+T-(S+T)=S.
\]
Both equations contradict Lemma~\ref{lem:positivity}. The rank-one assumption is impossible, and rank zero was already excluded in Remark~\ref{rem:rank-zero}. This proves Theorem~\ref{thm:main} and Corollary~\ref{cor:cuboid-rank}.


\section{Conclusions}\label{sec:conclusions}

The interchange used in the proof of Paulsen--West Theorem~4.2 is false at even indices, as the exact examples \eqref{eq:example-six} and \eqref{eq:example-five} demonstrate. Correcting its parity dependence leaves the pattern \eqref{eq:residual}. The missing case is settled by the unshifted-abscissa identity \eqref{eq:raw}, a primitive divisor at the largest index, and the positivity obstruction of Lemma~\ref{lem:positivity}.

The resulting conclusion is an emptiness statement: a congruent-number curve of rank at most one supports no positive cuboid configuration. In the original variables, a perfect cuboid would necessarily satisfy
\[
 \operatorname{rank} E_{\operatorname{sf}(abcg/2)}(\mathbb{Q})\ge2.
\]
Moreover, its three associated points must themselves generate a subgroup of rank at least two. This last formulation also excludes configurations contained in a subgroup of rank one inside a curve of higher rank.

The cuboid correspondence and the ambient-rank claim are due to Paulsen and West. The contribution here is a complete argument for the parity case omitted by their proof, together with the subgroup formulation that the argument directly establishes. It does not exclude configurations of generated rank at least two and therefore does not resolve the existence of perfect cuboids. No form of the Birch--Swinnerton-Dyer conjecture or the generalized Riemann hypothesis is assumed.

\section*{Acknowledgments}
The author thanks Randall L. Rathbun for helpful correspondence concerning the Paulsen--West paper and assistance with independent checks of the arguments and proofs.

\section*{AI assistance}
OpenAI GPT-6 Astra assisted in finding the explicit counterexamples to the claimed interchange of the sequences $A$ and $B$, developing and checking the corrected proof, and preparing the exposition.


\end{document}